\documentclass[11pt,reqno]{amsart}

\usepackage{amsmath,amssymb,amsthm,mathtools}
\usepackage{booktabs}
\usepackage{float}  
\usepackage{enumitem}
\usepackage[margin=1.15in]{geometry}
\usepackage[colorlinks=true,linkcolor=blue,citecolor=blue,urlcolor=blue]{hyperref}
\usepackage{bm}
\IfFileExists{bbm.sty}{\usepackage{bbm}}{\providecommand{\mathbbm}[1]{\mathbf{##1}}}
\newtheorem{theorem}{Theorem}[section]
\newtheorem{proposition}[theorem]{Proposition}
\newtheorem{lemma}[theorem]{Lemma}
\newtheorem{corollary}[theorem]{Corollary}
\theoremstyle{definition}
\newtheorem{definition}[theorem]{Definition}
\newtheorem{example}[theorem]{Example}
\theoremstyle{remark}
\newtheorem{remark}[theorem]{Remark}

\newcommand{\Rmax}{\mathbb{R}_{\max}}
\newcommand{\K}{\mathbb{K}}
\newcommand{\val}{\operatorname{val}}
\newcommand{\per}{\operatorname{per}}
\newcommand{\TP}{\mathsf{TP}}
\newcommand{\TN}{\mathsf{TN}}
\newcommand{\TPT}{\mathsf{TP}^{\mathrm{trop}}}
\newcommand{\TNT}{\mathsf{TN}^{\mathrm{trop}}}

\newcommand{\lc}{\operatorname{lc}}
\newcommand{\NP}{\operatorname{NP}^{+}}

\title[Spectral Equality for Tropical TN Matrices]{Spectral Equality for
Tropical Totally Nonnegative Matrices:\\
Dominant Cancellation and Newton Polygons}

\author[M. R. Saeidi Dahaki]{Mohammadrahim Saeidi Dahaki}
\address{Department of Mathematical Sciences, Isfahan University of Technology,
Isfahan 84156-83111, Iran}
\email{mrs.saeidi78@gmail.com}

\author[S. M. Manjegani]{Seyed Mahmoud Manjegani}
\address{Department of Mathematical Sciences, Isfahan University of Technology,
Isfahan 84156-83111, Iran}
\address{Department of Mathematics and Statistics, University of Regina,
Regina, Saskatchewan S4S 0A2, Canada}
\email{manjgani@iut.ac.ir}
\email{manjeg1s@uregina.ca}
\thanks{Corresponding author: S. M. Manjegani.}


\begin{document}

\begin{abstract}
Let $\mathbf A$ be a totally nonnegative matrix over a real closed
non-Archimedean valued field, and suppose that its valuation $A$ is tropically
totally nonnegative.  Although the valuations of the classical eigenvalues are
weakly majorized by the tropical eigenvalues of $A$, equality can fail through
cancellation.  We give a necessary and sufficient condition for equality.  The
only obstruction is dominant cancellation in the principal minors that give
the essential coefficients of the tropical characteristic polynomial.  In
other words, equality holds exactly when the classical and tropical Newton
polygons agree at the breakpoints determined by the diagonal entries.  We
describe the exceptional lifts by explicit polynomial equations in the
leading coefficients.  We also study these equations on each irreducible
component of the admissible parameter space.  Finally, we give a complete
$3\times3$ criterion and a family that contains both ordinary lifts with
spectral equality and exceptional totally nonnegative lifts without it.
\end{abstract}

\maketitle

\section{Introduction}

The max-plus semifield is
\[
 \Rmax=\mathbb{R}\cup\{-\infty\},
 \qquad
 a\oplus b=\max\{a,b\},
 \qquad
 a\odot b=a+b.
\]
Its additive and multiplicative identities are, respectively,
$\mathbbm{o}=-\infty$ and \(\mathbbm{1}=0\).  Matrices with entries in \(\Rmax\)
are called tropical matrices; for background on max-plus algebra and its
matrix theory, see Butkovi\v{c}~\cite{Butkovic}.  Tropical analogues of total positivity and
total nonnegativity were developed systematically by Gaubert and
Niv~\cite{GaubertNiv}; their work connects tropical minors with parity,
non-Archimedean valuation, Monge matrices, and tropical spectral theory.

Our goal is to find exactly when the known spectral majorization becomes an
equality for totally nonnegative lifts.  We first recall the facts that we need
about valued fields, tropical total nonnegativity, Monge matrices, and tropical
eigenvalues.  We then state weak majorization and its logarithmic form.  The
main result shows that equality holds when there is no dominant cancellation
at the breakpoints of the tropical Newton polygon.  Tropical
sign-nonsingularity gives a useful sufficient condition, but it is not
necessary.

Akian, Bapat, and Gaubert proved generic equality between the
valuations of classical eigenvalues and tropical roots for matrix
polynomials with unrestricted leading coefficients
\cite[Theorem~7.4]{AkianBapatGaubert}. Their result is stated in the
min-plus setting. By using
\(\operatorname{val}=-\nu\), it can be written in the max-plus setting
used in this paper. Gaubert and Niv also proved spectral equality for
totally positive lifts whose tropicalizations are tropically totally
positive; see \cite[Proposition~6.5]{GaubertNiv}.

Here we consider totally nonnegative lifts of a fixed tropically totally
nonnegative matrix \(A\). Thus our setting lies between the totally
positive case of Gaubert and Niv and the unrestricted setting of Akian,
Bapat, and Gaubert. In this setting, spectral equality does not always
hold. It may fail because of cancellation among the dominant terms,
even when the lift is totally nonnegative.

Our main result gives a necessary and sufficient condition for spectral
equality. It shows that we only need to check the strict vertices of the
tropical Newton polygon. At each such vertex, the obstruction is the
cancellation of the dominant terms in the relevant principal minors.
We describe the exceptional lifts by polynomial equations in the
leading coefficients. We also study these equations on each
irreducible component of the admissible parameter space. When the
component ideals are known, this gives an ideal-membership test.
Finally, for \(3\times3\) matrices, we give a complete explicit
criterion and a non-diagonal family containing both totally
nonnegative lifts with spectral equality and exceptional lifts for
which equality fails.

Throughout the paper, bold letters such as \(\mathbf A\) denote matrices over
the valued field \(\K\), while ordinary letters such as \(A\) denote their
tropicalizations.

\section{Valuations and tropical total nonnegativity}
We first introduce the non-Archimedean setting and tropical total
nonnegativity.  We then give the Monge characterization and connect it to
totally nonnegative lifts over the valued field.
In this section, we use the same notation as in~\cite{GaubertNiv}.

Let \(\Lambda\subset\mathbb{R}\) be empty, finite, or the set of elements of a
sequence decreasing to \(-\infty\). A generalized Puiseux series has the form
\[
 \boldsymbol f=\sum_{\lambda\in\Lambda}a_\lambda t^\lambda,
 \qquad a_\lambda\in\mathbb{R}\setminus\{0\}.
\]
The field of such series will be denoted by \(\mathbb{K}\); see
Markwig~\cite{Markwig}. 
Its algebraic closure is the quadratic extension 
\(\overline{\mathbb{K}}=\mathbb{K}(\sqrt{-1})\), which is algebraically
closed, as shown in~\cite{Markwig}.  The valuation extends naturally to
\(\overline{\mathbb{K}}\), and we use the same notation \(\val\) for this
extension.

For a nonzero series, define
\[
 \val(\boldsymbol f)=\sup\{\lambda:a_\lambda\ne0\},
 \qquad
 \val(0)=-\infty.
\]

\begin{definition}
A non-Archimedean valuation on a field \(\K\) is a map
\(\val:\K\to\Rmax\) such that, for all \(f,g\in\K\),
\begin{enumerate}[label=\textup{(\roman*)}]
 \item \(\val(f+g)\leq\max\{\val(f),\val(g)\}\);
 \item \(\val(fg)=\val(f)+\val(g)\);
 \item \(\val(f)=-\infty\) if and only if \(f=0\).
\end{enumerate}
\end{definition}

The leading coefficient of a nonzero generalized Puiseux series is the
coefficient corresponding to its largest exponent.  The series is positive if
its leading coefficient is positive.  We write \(\K_{>0}\) for the positive
series and \(\K_{\geq0}=\K_{>0}\cup\{0\}\).

The next elementary observation records how the valuation behaves on
nonnegative series.
\begin{proposition}
The restriction \(\val:\K_{\geq0}\to\Rmax\) is a semifield
homomorphism.  In particular,
\[
 \val(f+g)=\val(f)\oplus\val(g),
 \qquad
 \val(fg)=\val(f)\odot\val(g)
 \quad (f,g\in\K_{\geq0}).
\]
Also, \(\val(f^{-1})=-\val(f)\) for every nonzero \(f\in\K\).
\end{proposition}

\begin{proof}
Positive leading terms cannot cancel, which gives the first equality.  The
second is the multiplicative property of the valuation.  Finally,
\(0=\val(1)=\val(f)+\val(f^{-1})\).
\end{proof}

If \(\mathbf A=(\mathbf A_{ij})\in\K^{m\times n}\), its entrywise
tropicalization is
\[
 A=\val(\mathbf A):=(\val(\mathbf A_{ij}))\in\Rmax^{m\times n}.
\]
Conversely, \(\mathbf A\) is called a lift of \(A\) if this equality holds.

For \(A\in\Rmax^{n\times n}\), its tropical permanent is
\[
 \per(A)=\bigoplus_{\sigma\in S_n}
       \bigodot_{i=1}^n A_{i,\sigma(i)}
 =\max_{\sigma\in S_n}\sum_{i=1}^n A_{i,\sigma(i)}.
\]
The sum associated with \(\sigma\) is called the weight of \(\sigma\).
A square tropical matrix is \emph{tropically sign-nonsingular} if its
permanent is finite and all maximizing permutations have the same parity.

The following proposition gives a basic condition that prevents cancellation
among the dominant terms of a determinant.
\begin{proposition}\label{prop:det-per}
Let \(A\in\Rmax^{n\times n}\) be tropically sign-nonsingular, and let
\(\mathbf A\in\K^{n\times n}\) be a lift whose nonzero entries have leading
coefficients of one common sign.  Then
\[
 \val(\det\mathbf A)=\per(A).
\]
\end{proposition}

\begin{proof}
Expand \(\det\mathbf A\) over \(S_n\).  Terms corresponding to maximizing
permutations have valuation \(\per(A)\).  Since these permutations have the
same parity and the relevant leading coefficients have one common sign, their
leading terms have a common sign and cannot cancel.  If the common entry sign
is negative, every permutation product acquires the same factor $(-1)^n$.
Every other term has strictly
smaller valuation.  The conclusion follows.
\end{proof}

\begin{remark}
The common-sign assumption includes both positive and negative leading
coefficients.  The exact condition is weaker: the sum of the leading
coefficients of the maximizing determinant terms must be nonzero; see
Lemma~\ref{lem:dominant-cancellation}.
\end{remark}

\begin{definition}
A tropical minor is \emph{tropically positive} if all its maximizing
permutations are even.  It is \emph{tropically nonnegative} if its maximizing
permutations are all even or if the minor is tropically sign-singular.  A
matrix is tropically totally positive, written \(A\in\TPT\), if its entries
are finite and all its minors are tropically positive.  It is tropically
totally nonnegative, written \(A\in\TNT\), if its entries
are in $\Rmax$ and all its minors are tropically
nonnegative.
\end{definition}

The corresponding classical classes are denoted by \(\TP\) and \(\TN\); see
Fallat and Johnson~\cite{FallatJohnson}.  We use \(\TPT_k\) and \(\TNT_k\)
for the tropical classes defined by minors of order at most \(k\).

The following reduction to minors of order two is
\cite[Theorem~3.4]{GaubertNiv}.
\begin{theorem}\label{thm:two-by-two}
For tropical matrices of arbitrary size,
\[
 \TPT=\TPT_2,
 \qquad
 \TNT=\TNT_2.
\]
\end{theorem}

Thus tropical total positivity and nonnegativity can be tested using only
\(2\times2\) minors.  This is one of the principal simplifications of the
tropical theory.

A real $m\times n$ matrix $A$ is a \emph{Monge matrix} if
\begin{equation}\label{eq:Monge}
 A_{ij}+A_{i'j'}\geq A_{ij'}+A_{i'j}
 \qquad(i<i',\ j<j').
\end{equation}
By Theorem~\ref{thm:two-by-two}, a finite tropical matrix is tropically
totally nonnegative if and only if it is Monge.  Thus
\eqref{eq:Monge} provides the numerical test used below.  We will also use
the following valuation relation without recalling the staircase and support
descriptions, which play no role in the spectral arguments; see
\cite[Section~5]{GaubertNiv}.

The following comparison of the classical and tropical classes, together with
the lift characterization of \(\mathrm{TP}^{\mathrm{trop}}\), follows from
\cite[Theorems~5.2 and 5.3]{GaubertNiv}.
\begin{theorem}\label{thm:valuation-classes}
 For matrices of size at least \(2\times2\),
\[
 \TPT\subsetneqq\val(\TP)
 \subsetneqq\val(\TN)
 \subseteq\TNT.
\]
Also, a tropical matrix $A$ belongs to $\TPT$ if and only if every lift
of $A$ with positive leading coefficients is totally positive.
\end{theorem}
\section{Tropical characteristic polynomials and eigenvalues}
In this section, we introduce tropical characteristic polynomials and tropical
eigenvalues. We describe their properties and compare them with
the eigenvalues of totally nonnegative lifts. 

Following \cite[Section~6]{GaubertNiv}, let \(A\in\Rmax^{n\times n}\). The tropical characteristic polynomial of $A$ is defined by
\begin{equation}\label{xx}
 f_A(x)=\per(x\odot I\oplus A)
       =x^{\odot n}\oplus
        \bigoplus_{i=0}^{n-1}a_{n-i}\odot x^{\odot i}, 
\end{equation}
where \(a_k\) is the maximum tropical permanent among all \(k\times k\)
principal submatrices of \(A\), and \(a_0=\mathbbm{1}\).

Just as the fundamental theorem of algebra allows the factorization of
classical polynomials into linear factors, tropical polynomials admit an
analogous factorization.  Specifically, Agudelo and
Lorscheid~\cite{Ag} prove that the irreducible tropical polynomials are
precisely the linear tropical polynomials, and every tropical polynomial has
a unique factorization into irreducibles.
\begin{definition}
A \emph{max-plus eigenvalue} is a scalar \(\lambda\in\Rmax\) for which there
exists a nonzero tropical vector \(v\) satisfying
\[
 A\odot v=\lambda\odot v.
\]
\end{definition}
Tropical roots of a tropical polynomial can be defined by several equivalent
approaches.  In this work, we adopt the approach from~\cite{peper}.
\begin{definition}
Let
$
f_A(x)
$
be the tropical characteristic polynomial of $A$.
A number \(r\in\mathbb{R}\) is called a \emph{tropical root} of \(f_A\) if the maximum in
\eqref{xx} is attained at least twice; that is, there exist distinct indices
\(i,j\in\{0,\dots,n\}\) such that
\[
a_{n-i}+i r=a_{n-j}+j r\ge a_{n-k}+k r
\qquad\text{for all } k\in\{0,\dots,n\}.
\]
The multiplicity of the root \(r\) is
\[
\operatorname{mult}(r)
=
\max\Bigl\{|i-j|\;\Bigm|\;
a_{n-i}+i r=a_{n-j}+j r=f_A(r)
\Bigr\}.
\]
We also count \(-\infty\) as a root with multiplicity \(l\) whenever
\[
a_n=a_{n-1}=\cdots=a_{n-l+1}=-\infty,
\qquad a_{n-l}\neq-\infty.
\]
\end{definition}
The next result of Gaubert and Niv shows that, for a matrix in \(\TNT\), both
the tropical characteristic coefficients and the tropical eigenvalues are
completely determined by the diagonal entries; see
\cite[Propositions~6.1 and~6.3]{GaubertNiv}.
\begin{theorem}\label{thm:diagonal-spectrum}
Let \(A\in\TNT\).  Then the coefficient \(a_k\) of \(f_A\) is the tropical
product of the \(k\) largest diagonal entries of \(A\).  Therefore, the
tropical eigenvalues of \(A\), counted with multiplicity, are its
diagonal entries.
\end{theorem}

The following example distinguishes the algebraic tropical roots of the
characteristic polynomial from the max-plus eigenvalues of the matrix.
\begin{example}
Let
\[
A=\begin{pmatrix}
1&-\infty&-\infty\\
2&-\infty&-\infty\\
-\infty&4&5
\end{pmatrix}.
\]
The tropical characteristic polynomial is
\[
f_A(x)=x^{\odot 3}\oplus 5\odot x^{\odot 2}\oplus 6\odot x\oplus (-\infty).
\]

The tropical roots are \(5\), \(1\), and \(-\infty\):
\begin{itemize}
    \item At \(x=5\), the maximum \(15\) is attained by \(3x\) and \(2x+5\).
    \item At \(x=1\), the maximum \(7\) is attained by \(2x+5\) and \(x+6\).
    \item Since \(a_3=-\infty\) and \(a_2=6\ne -\infty\), \(-\infty\) is a root of multiplicity \(1\).
\end{itemize}
On the other hand, only \(5\) is a max-plus eigenvalue; neither \(1\) nor \(-\infty\) is.
\end{example}
We next compare classical eigenvalues over \(\K\) with tropical
eigenvalues. Related asymptotic formulae connecting max-algebraic eigenvalues
with distinguished classical eigenvalues of nonnegative matrices were obtained
in~\cite{ManjeganiPointSpectrum}. Here we use the following result of Gaubert
and Niv~\cite[Proposition~6.5]{GaubertNiv}.
\begin{theorem}\label{thm:eigenvalue-valuation}
Let \(\mathbf A\in\TP(\K)\), put \(A=\val(\mathbf A)\), and assume
\(A\in\TPT\).  Write
\[
 p_{\mathbf A}(\lambda)
 =\lambda^n+\sum_{k=1}^{n}(-1)^k
   \boldsymbol\alpha_k\lambda^{n-k}
\]
and let $a_k$ be the coefficient of $x^{\odot(n-k)}$ in $f_A$, with
$\boldsymbol\alpha_0=1$ and $a_0=0$.  Let
\(\rho_1,\ldots,\rho_n\) be the eigenvalues of \(\mathbf A\), counted with
algebraic multiplicity and indexed so that
\[
 \val(\rho_1)\geq\cdots\geq\val(\rho_n),
\]
and let
\(\eta_1\geq\cdots\geq\eta_n\) be the tropical eigenvalues of
\(A\).  Then
\[
 \val(\boldsymbol\alpha_k)=a_k,
 \qquad k=0,\ldots,n,
\]
and
\[
 \val(\rho_i)=\eta_i,
 \qquad i=1,\ldots,n.
\]
\end{theorem}

For a general lift, cancellation may occur in the coefficients of the
classical characteristic polynomial.  Therefore, for a general lift we only
have \(\val(\boldsymbol\alpha_k)\leq a_k\).

\section{Majorization and a logarithmic formulation}
Majorization is the natural tool for comparing the two spectra; see
Marshall, Olkin, and Arnold~\cite{MarshallOlkinArnold}. We recall the weak
majorization inequality for eigenvalue valuations, and then state its
multiplicative form, weak log-majorization.

Stochastic matrices and vector majorization formulated directly in max algebra
were studied in~\cite{ManjeganiParsa}. That notion is different from the
additive weak majorization used below: here we compare the ordinary partial
sums of the decreasingly ordered valuations.

For \(x=(x_1,\ldots,x_n)\in\Rmax^n\), write
\(x_{[1]}\geq\cdots\geq x_{[n]}\) for its decreasing rearrangement.

\begin{definition}
For \(x,y\in\Rmax^n\), we say that \(x\) is weakly majorized by \(y\),
and write \(x\prec_w y\), if
\[
 \sum_{i=1}^k x_{[i]}\leq\sum_{i=1}^k y_{[i]},
 \qquad k=1,\ldots,n.
\]
If equality also holds for the total sums, we write \(x\prec y\).
Here sums are taken in the extended real line, with
\(-\infty+a=-\infty\) for every finite \(a\).
\end{definition}

For nonnegative vectors \(x,y\in[0,\infty)^n\), weak log-majorization is
defined by
\[
 x\prec_{w,\log}y
 \quad\Longleftrightarrow\quad
 \prod_{i=1}^k x_{[i]}\leq\prod_{i=1}^k y_{[i]}
 \quad(k=1,\ldots,n).
\]
If equality holds for \(k=n\), we write \(x\prec_{\log}y\); see
\cite{MarshallOlkinArnold,AkianGaubertSharify} for background and related
spectral log-majorization results.
Allowing zero coordinates is essential here: a zero eigenvalue has valuation
$-\infty$ and non-Archimedean modulus zero.
Throughout this formulation, we use the conventions
$e^{-\infty}=0$ and that a product containing a zero factor is zero.

The following result is the max-plus form of
\cite[Theorem~4.4]{AkianBapatGaubert}; it follows from the min-plus
formulation there by the sign change \(\val=-\nu\). See also
Gaubert and Niv~\cite[Remark~6.6]{GaubertNiv}.

\begin{theorem}\label{thm:valuation-majorization}
Let \(\mathbf A\in\mathbb{K}^{n\times n}\), set
\(A=\val(\mathbf A)\), and let
\(\rho_1,\ldots,\rho_n\in\overline{\mathbb{K}}\) be the classical
eigenvalues of \(\mathbf A\), counted with algebraic multiplicity.
Let \(\eta_1,\ldots,\eta_n\) be the algebraic tropical eigenvalues
of \(A\), counted with multiplicity. Then
\[
 \bigl(\val(\rho_1),\ldots,\val(\rho_n)\bigr)
 \prec_w
 \bigl(\eta_1,\ldots,\eta_n\bigr).
\]
If, in addition, the hypotheses of
Theorem~\ref{thm:eigenvalue-valuation} hold, then the eigenvalues
can be indexed so that
\[
 \val(\rho_1)\geq\cdots\geq\val(\rho_n),
 \qquad
 \eta_1\geq\cdots\geq\eta_n,
\]
and
\[
 \val(\rho_i)=\eta_i,
 \qquad i=1,\ldots,n.
\]
\end{theorem}
For \(x\in\K\), define the non-Archimedean modulus
\[
 |x|_{\val}=\begin{cases}
 e^{\val(x)},&x\ne0,\\
 0,&x=0.
 \end{cases}
\]
It satisfies \(|xy|_{\val}=|x|_{\val}|y|_{\val}\) and
\(|x+y|_{\val}\leq\max\{|x|_{\val},|y|_{\val}\}\).
We use the same definition for elements of \(\overline{\mathbb{K}}\).

Applying the exponential map to the preceding theorem gives its multiplicative
form.
\begin{corollary}\label{cor:log-majorization}
Under the hypotheses of Theorem~\ref{thm:valuation-majorization},
\[
 \bigl(|\rho_1|_{\val},\ldots,|\rho_n|_{\val}\bigr)
 \prec_{w,\log}
 \bigl(e^{\eta_1},\ldots,e^{\eta_n}\bigr).
\]
If \(\mathbf A\) is totally positive and \(A\in\TPT\), then
\[
 |\rho_i|_{\val}=e^{\eta_i},
 \qquad i=1,\ldots,n,
\]
after a common decreasing ordering.  In particular, equality holds in every
partial-product inequality.
\end{corollary}

\begin{proof}
For \(k=1,\ldots,n\), Theorem~\ref{thm:valuation-majorization} gives
\[
 \sum_{i=1}^k\val(\rho_{[i]})
 \leq
 \sum_{i=1}^k\eta_i.
\]
Exponentiating yields
\[
 \prod_{i=1}^k|\rho_{[i]}|_{\val}
 =\exp\!\left(\sum_{i=1}^k\val(\rho_{[i]})\right)
 \leq
 \exp\!\left(\sum_{i=1}^k\eta_i\right)
 =\prod_{i=1}^k e^{\eta_i}.
\]
This is exactly weak log-majorization.  The equality statement follows from
Theorem~\ref{thm:eigenvalue-valuation}.
\end{proof}

\begin{remark}
Corollary~\ref{cor:log-majorization} is the multiplicative form of
Theorem~\ref{thm:valuation-majorization}.  It follows by applying the
order-preserving exponential map.  This also explains the link between
additive tropical majorization and classical log-majorization.
\end{remark}

\section{A spectral equality characterization}
\label{sec:spectral-characterization}
This section contains our main characterization of spectral equality.  By
analyzing dominant cancellation in principal minors through their leading
coefficients, we obtain a Newton-polygon criterion.  We then describe the
exceptional locus algebraically and give an explicit criterion in dimension
three.  Throughout, \(\mathbb{K}\) is the real closed field of generalized
Puiseux series introduced earlier.

We use the max-type valuation convention
\[
 \val(t^\lambda)=\lambda,
 \qquad
 \val(fg)=\val(f)+\val(g),
 \qquad
 \val(f+g)\leq\max\{\val(f),\val(g)\}.
\]
Every nonzero $f\in\K$ has a unique leading-term expansion
\[
 f=c t^\lambda+\text{terms of valuation strictly smaller than $\lambda$},
 \qquad
 \lambda=\val(f),
\]
where $c\in\mathbb R\setminus\{0\}$.  We denote this leading coefficient by
$\lc(f)=c$.  If $f>0$, then $\lc(f)>0$.  Bold letters denote matrices over
$\K$, and ordinary letters denote their entrywise tropicalizations.

We use the Newton polygon theorem in its valued-field form; see, for example,
\cite[Section~2]{AkianBapatGaubert}.  With the max-type valuation used here, the relevant
Newton polygon is the \emph{upper} Newton polygon.  Thus, if
\[
 p(\lambda)=\sum_{j=0}^{n}c_j\lambda^j,
 \qquad c_0c_n\neq0,
\]
then $\NP(p)$ is the upper boundary of the convex hull of
\[
 \{(j,\val(c_j)):c_j\neq0\}.
\]
A segment of slope $m$ and horizontal length $\ell$ corresponds to exactly
$\ell$ roots, counted with multiplicity, of valuation $-m$.

Let
\[
 \mathbf A\in\TN(\K),
 \qquad
 A=\val(\mathbf A)\in(\TNT)^{n\times n},
\]
and assume that all diagonal entries of $A$ are finite.  Write them in
decreasing order as
\[
 \eta_1\geq\eta_2\geq\cdots\geq\eta_n.
\]
The tropical characteristic polynomial is
\[
 f_A(x)=x^{\odot n}\oplus
 \bigoplus_{k=1}^{n}a_k\odot x^{\odot(n-k)},
\]
where tropical total nonnegativity gives
\begin{equation}\label{eq:ak-diagonal}
 a_0=0,
 \qquad
 a_k=\eta_1+\cdots+\eta_k,
 \quad k=1,\ldots,n.
\end{equation}
In other words,
\begin{equation}\label{eq:ak-principal}
 a_k=max_{\substack{I\subseteq[n]\\|I|=k}}
       \per(A[I,I]).
\end{equation}
We write the classical characteristic polynomial as
\begin{equation}\label{eq:classical-characteristic}
 p_{\mathbf A}(\lambda)=\det(\lambda I-\mathbf A)
 =\lambda^n+
 \sum_{k=1}^{n}(-1)^k\boldsymbol\alpha_k\lambda^{n-k},
\end{equation}
where
\begin{equation}\label{eq:alpha-principal-sum}
 \boldsymbol\alpha_k
 =\sum_{\substack{I\subseteq[n]\\|I|=k}}
   \det\mathbf A[I,I].
\end{equation}
We also put $\boldsymbol\alpha_0=1$.

\subsection{Dominant cancellation in a determinant}

Let
\[
 \mathbf B=(\mathbf b_{ij})\in\K^{m\times m},
 \qquad B=\val(\mathbf B),
\]
and assume that $\per(B)>-\infty$.  For every nonzero entry write
\[
 \mathbf b_{ij}
 =c_{ij}t^{B_{ij}}
 +\text{terms of valuation smaller than $B_{ij}$},
 \qquad c_{ij}=\lc(\mathbf b_{ij}).
\]
Define the set of maximizing permutations by
\[
 \mathcal M(B)
 =\left\{
   \sigma\in S_m:
   \sum_{i=1}^{m}B_{i,\sigma(i)}=\per(B)
  \right\}.
\]

\begin{definition}\label{def:dominant-coefficient}
The \emph{dominant determinant coefficient} of the lift $\mathbf B$ over $B$
is
\begin{equation}\label{eq:dominant-coefficient}
 \Delta_B(\mathbf B)
 =\sum_{\sigma\in\mathcal M(B)}
   \operatorname{sgn}(\sigma)
   \prod_{i=1}^{m}c_{i,\sigma(i)}.
\end{equation}
Thus $\Delta_B(\mathbf B)$ is the coefficient of $t^{\per(B)}$ in the
determinant expansion of $\mathbf B$.
\end{definition}

The notation records an important distinction: $\Delta_B(\mathbf B)$ depends
not only on the tropical matrix $B$, but also on the leading coefficients of
the chosen lift $\mathbf B$.

The next lemma gives the precise criterion for dominant cancellation in a
determinant.
\begin{lemma}
\label{lem:dominant-cancellation}
Let $\mathbf B\in\K^{m\times m}$, put $B=\val(\mathbf B)$, and assume that
$\per(B)>-\infty$.  Then
\begin{equation}\label{eq:det-per-upper}
 \val(\det\mathbf B)\leq\per(B).
\end{equation}
The following statements are equivalent:
\begin{enumerate}[label=\textup{(\roman*)}]
 \item $\val(\det\mathbf B)=\per(B)$;
 \item the terms of maximal valuation in the determinant expansion do not
 cancel completely;
 \item $\Delta_B(\mathbf B)\neq0$.
\end{enumerate}
If $\Delta_B(\mathbf B)=0$, then either $\det\mathbf B=0$, or
\[
 \val(\det\mathbf B)<\per(B).
\]
\end{lemma}

\begin{proof}
Expanding the determinant gives
\begin{equation}\label{eq:det-expansion}
 \det\mathbf B
 =\sum_{\sigma\in S_m}
  \operatorname{sgn}(\sigma)
  \prod_{i=1}^{m}\mathbf b_{i,\sigma(i)}.
\end{equation}
For every nonzero permutation term, multiplicativity of the valuation yields
\[
 \val\!\left(\prod_{i=1}^{m}\mathbf b_{i,\sigma(i)}\right)
 =\sum_{i=1}^{m}B_{i,\sigma(i)}
 \leq\per(B).
\]
The non-Archimedean triangle inequality applied to
\eqref{eq:det-expansion} proves \eqref{eq:det-per-upper}.

If $\sigma\in\mathcal M(B)$, then
\[
 \prod_{i=1}^{m}\mathbf b_{i,\sigma(i)}
 =\left(\prod_{i=1}^{m}c_{i,\sigma(i)}\right)t^{\per(B)}
 +\text{terms of smaller valuation}.
\]
All permutations outside $\mathcal M(B)$ contribute only terms of valuation
strictly smaller than $\per(B)$.  Therefore,
\begin{equation}\label{eq:det-leading-expansion}
 \det\mathbf B
 =\Delta_B(\mathbf B)t^{\per(B)}
 +\text{terms of valuation smaller than $\per(B)$}.
\end{equation}
Thus the determinant has valuation $\per(B)$ exactly when
$\Delta_B(\mathbf B)\neq0$.  If this coefficient vanishes, then the determinant
is either zero or has strictly smaller valuation.
\end{proof}

The following example illustrates the dominant-coefficient calculation when
several permutations attain the tropical permanent.
\begin{example}
\label{ex:general-dominant-coefficient}
Let
\[
B=
\begin{pmatrix}
1&2&3\\
2&1&4\\
3&4&1
\end{pmatrix}.
\]
Consider the lift
\[
\mathbf{B}=
\begin{pmatrix}
2t+5 & 3t^2+t & t^3+4t^2\\
4t^2-t & 5t+2 & 6t^4+t^2\\
7t^3+t^2 & 8t^4+t^3 & 9t+1
\end{pmatrix}.
\]
Then
$
\val(\mathbf{B})=B.
$
The leading coefficients \(c_{ij}=\lc(\mathbf{b}_{ij})\) are
\[
(c_{ij})=
\begin{pmatrix}
2&3&1\\
4&5&6\\
7&8&9
\end{pmatrix}.
\]
The weights of all permutations are
\[
\begin{array}{c|cccccc}
\sigma & \mathrm{id} & (12) & (13) & (23) & (123) & (132)\\
\hline
\sum_i B_{i,\sigma(i)} & 3 & 5 & 7 & 9 & 9 & 9
\end{array}.
\]
The tropical permanent is
\[
\per(B)=\max_{\sigma\in S_3}\sum_{i=1}^3 B_{i,\sigma(i)}=9.
\]
Thus
$
\mathcal{M}(B)=\{(23),(123),(132)\}.
$
The dominant determinant coefficient is
\[
\Delta_B(\mathbf{B})
=
\sum_{\sigma\in\mathcal{M}(B)}
\operatorname{sgn}(\sigma)\prod_{i=1}^3 c_{i,\sigma(i)}=
-c_{11}c_{23}c_{32}
+c_{12}c_{23}c_{31}
+c_{13}c_{21}c_{32}.
\]
Substituting the leading coefficients,
\[
\Delta_B(\mathbf{B})
=
-2\cdot 6\cdot 8
+3\cdot 6\cdot 7
+1\cdot 4\cdot 8
=
62.
\]
Thus the coefficient of \(t^{\per(B)}=t^9\) in \(\det(\mathbf{B})\) is \(62\). Therefore,
\[
\det(\mathbf{B})=62t^9+\text{terms of valuation }<9,
\]
and
\[
\val(\det(\mathbf{B}))=9=\per(B).
\]
This example only illustrates the dominant-coefficient calculation; no total
nonnegativity assumption is made on $\mathbf B$.
\end{example}

For $k=1,\ldots,n$, define
\begin{equation}\label{eq:maximizing-principal-sets}
 \mathcal I_k(A)
 =\left\{
   I\subseteq[n]:|I|=k,\ \per(A[I,I])=a_k
  \right\}.
\end{equation}
Thus $\mathcal I_k(A)$ consists of the principal index sets that realize the
$k$th tropical characteristic coefficient.  It is nonempty: if $I_k$ selects
$k$ diagonal entries of largest value, then the identity permutation in
$A[I_k,I_k]$ has weight $a_k$, and \eqref{eq:ak-principal} gives
$\per(A[I_k,I_k])=a_k$.

The following proposition determines exactly when a classical characteristic
coefficient attains its tropical upper bound.
\begin{proposition}
\label{prop:coefficient-realization}
Let $\mathbf A$ be a totally nonnegative matrix. For every $k=1,\ldots,n$,
\begin{equation}\label{eq:coefficient-upper-bound}
 \val(\boldsymbol\alpha_k)\leq a_k.
\end{equation}
The following statements are equivalent:
\begin{enumerate}[label=\textup{(\roman*)}]
 \item $\val(\boldsymbol\alpha_k)=a_k$;
 \item there exists $I\in\mathcal I_k(A)$ such that
 \[
  \val\!\left(\det\mathbf A[I,I]\right)=a_k;
 \]
 \item there exists $I\in\mathcal I_k(A)$ such that
 \[
  \val\!\left(\det\mathbf A[I,I]\right)=\per(A[I,I]);
 \]
 \item there exists $I\in\mathcal I_k(A)$ such that
 \[
  \Delta_{A[I,I]}\!\left(\mathbf A[I,I]\right)\neq0.
 \]
\end{enumerate}
\end{proposition}

\begin{proof}
Because $\mathbf A$ is totally nonnegative, every principal minor in
\eqref{eq:alpha-principal-sum} is nonnegative.  Thus there is no cancellation
between different principal minors, and
\begin{equation}\label{eq:alpha-max-minors}
 \val(\boldsymbol\alpha_k)
 =\max_{\substack{I\subseteq[n]\\|I|=k}}
   \val\!\left(\det\mathbf A[I,I]\right),
\end{equation}
where zero minors are assigned valuation $-\infty$.  For every $I$ of
cardinality $k$, Lemma~\ref{lem:dominant-cancellation} and
\eqref{eq:ak-principal} give
\begin{equation}\label{eq:minor-chain}
 \val\!\left(\det\mathbf A[I,I]\right)
 \leq\per(A[I,I])\leq a_k.
\end{equation}
Together with \eqref{eq:alpha-max-minors}, this proves
\eqref{eq:coefficient-upper-bound}.

Suppose $\val(\boldsymbol\alpha_k)=a_k$.  By
\eqref{eq:alpha-max-minors}, some $I$ of cardinality $k$ satisfies
\[
 \val\!\left(\det\mathbf A[I,I]\right)=a_k.
\]
The two inequalities in \eqref{eq:minor-chain} must then both be equalities, so
$I\in\mathcal I_k(A)$.  Conversely, the existence of such an $I$ and
\eqref{eq:alpha-max-minors} imply
$\val(\boldsymbol\alpha_k)\geq a_k$; the reverse inequality is
\eqref{eq:coefficient-upper-bound}.  This shows that (i) and
(ii) are equivalent.

For $I\in\mathcal I_k(A)$ we have $\per(A[I,I])=a_k$, so (ii) and (iii) are
equivalent.  Finally, Lemma~\ref{lem:dominant-cancellation} gives the equivalence
of (iii) and (iv).
\end{proof}

\subsection{Essential vertices of the tropical Newton polygon}

Define the breakpoint set
\begin{equation}\label{eq:breakpoint-set}
 \mathcal B(A)
 =\{k\in\{1,\ldots,n-1\}:\eta_k>\eta_{k+1}\}\cup\{n\}.
\end{equation}

The next lemma identifies the coefficient points that are essential for the
upper tropical Newton polygon.
\begin{lemma}
\label{lem:essential-vertices}
Put
\[
 P_k=(n-k,a_k),
 \qquad k=0,\ldots,n.
\]
The polygonal line through $P_n,P_{n-1},\ldots,P_0$ is concave and is the upper
Newton polygon of $f_A$.  For $1\leq k\leq n-1$, the point $P_k$ is a strict
vertex if and only if $\eta_k>\eta_{k+1}$.  Thus, apart from the automatically
fixed endpoint $P_0=(n,0)$, the essential vertices are
\[
 P_k,\qquad k\in\mathcal B(A).
\]
\end{lemma}

\begin{proof}
The slope of the segment from $P_k$ to $P_{k-1}$ is
\begin{equation}\label{eq:tropical-segment-slope}
 \frac{a_{k-1}-a_k}{(n-k+1)-(n-k)}=-\eta_k.
\end{equation}
From left to right the successive slopes are
\[
 -\eta_n,-\eta_{n-1},\ldots,-\eta_1,
\]
which are nonincreasing.  Thus the polygonal line is concave.  At $P_k$, the
left and right slopes are $-\eta_{k+1}$ and $-\eta_k$, respectively.  They are
different exactly when $\eta_k>\eta_{k+1}$.  The endpoint assertion is
immediate.  This is the max-type form of the strict-vertex criterion for
formal tropical polynomials; compare
Akian--Bapat--Gaubert~\cite[Lemma~2.6]{AkianBapatGaubert}.  Their statement
uses the min-plus convention, so reversing the valuation order yields the
upper-polygon formulation used here.
\end{proof}

The following example shows why only strict vertices of the tropical Newton
polygon need to be tested.
\begin{example}\label{ex:non-strict-vertex}
Consider the tropical matrix
\[
A=\begin{pmatrix}
5&4&-\infty\\
3&3&2\\
-\infty&2&3
\end{pmatrix}\in\TNT.
\]
Its diagonal entries are \(5,3,3\), so
\[
 \eta_1=5,\qquad \eta_2=3,\qquad \eta_3=3.
\]
Since \(\eta_1>\eta_2=\eta_3\),
\[
\mathcal{B}(A)=\{1,3\}.
\]
The tropical coefficients are
\[
a_0=0,\qquad a_1=5,\qquad a_2=5+3=8,\qquad a_3=5+3+3=11.
\]
Thus the points of the upper Newton polygon are
\[
P_0=(3,0),\quad P_1=(2,5),\quad P_2=(1,8),\quad P_3=(0,11).
\]
From left to right, the slopes are
\[
 -\eta_3=-3,\qquad -\eta_2=-3,\qquad -\eta_1=-5,
\]
which are nonincreasing; hence the polygonal line is concave.
At \(P_2\), the left and right slopes are both \(-3\), so \(P_2\) is \emph{not} a strict vertex.  
Therefore the essential vertices are
\(P_3,P_1\), and \(P_0\),
corresponding to \(k\in\mathcal{B}(A)=\{1,3\}\) together with the fixed endpoint \(P_0\).  
\end{example}

We now relate the Newton polygons to spectral equality.
Proposition~\ref{prop:newton-polygon-equality} characterizes when the polygons
coincide, and Proposition~\ref{prop:newton-spectral-equality} shows that this
coincidence is equivalent to spectral equality.
\begin{proposition}
\label{prop:newton-polygon-equality}
The equality
\begin{equation}\label{eq:newton-polygon-equality}
 \NP(p_{\mathbf A})=\NP(f_A)
\end{equation}
holds if and only if
\begin{equation}\label{eq:breakpoint-coefficient-equality}
 \val(\boldsymbol\alpha_k)=a_k
 \qquad\text{for every }k\in\mathcal B(A).
\end{equation}
\end{proposition}

\begin{proof}
For $k=0$ we have $\val(\boldsymbol\alpha_0)=a_0=0$, and for
$k=1,\ldots,n$, Proposition~\ref{prop:coefficient-realization} gives
\begin{equation}\label{eq:classical-points-below}
 \val(\boldsymbol\alpha_k)\leq a_k.
\end{equation}
Thus every classical coefficient point
\[
 Q_k=(n-k,\val(\boldsymbol\alpha_k))
\]
lies on or below the corresponding tropical point $P_k$; if
$\boldsymbol\alpha_k=0$, the point $Q_k$ is omitted.

Suppose first that \eqref{eq:breakpoint-coefficient-equality} holds.  Then every
essential tropical vertex $P_k$, $k\in\mathcal B(A)$, is also a classical
coefficient point, and $P_0=(n,0)$ is common to the two diagrams because both
polynomials are monic.  All remaining classical points lie on or below the
tropical polygon by \eqref{eq:classical-points-below}; hence they cannot change
its upper boundary.  Therefore \eqref{eq:newton-polygon-equality} holds.

Conversely, assume \eqref{eq:newton-polygon-equality}.  If
$k\in\mathcal B(A)\cap\{1,\ldots,n-1\}$, then $P_k$ is a strict vertex of the
common upper convex hull.  A vertex of the convex hull of a finite set must be
one of the generating points.  Since there is only one classical coefficient
point with horizontal coordinate $n-k$, we have $Q_k=P_k$.  Thus
$\val(\boldsymbol\alpha_k)=a_k$.  The same conclusion for $k=n$ follows from
equality of the left endpoints at horizontal coordinate zero.  This gives
\eqref{eq:breakpoint-coefficient-equality}.
\end{proof}

\begin{proposition}
\label{prop:newton-spectral-equality}
Let $\rho_1,\ldots,\rho_n$ be the eigenvalues of $\mathbf A$, counted with
algebraic multiplicity.  Then the following statements are equivalent:
\begin{enumerate}[label=\textup{(\roman*)}]
 \item the eigenvalues can be ordered so that
 \[
  \val(\rho_i)=\eta_i,
  \qquad i=1,\ldots,n;
 \]
 \item $\NP(p_{\mathbf A})=\NP(f_A)$.
\end{enumerate}
\end{proposition}

\begin{proof}
The spectral theorem for totally nonnegative matrices shows that the
eigenvalues of $\mathbf A$ are real and nonnegative, counted with algebraic
multiplicity; see \cite[Section~5.2]{FallatJohnson}.  The corresponding
statement is first-order expressible for every fixed matrix size.  Its version
over the real closed field $\K$ therefore follows from completeness of the
theory of real closed fields; see
\cite[Theorem~3.3.15 and Corollary~3.3.16]{Marker}.  The Newton
polygon theorem identifies the multiset
\[
 \{\val(\rho_1),\ldots,\val(\rho_n)\}
\]
with the negatives of the slopes of $\NP(p_{\mathbf A})$, each repeated
according to the horizontal length of its segment.

By \eqref{eq:tropical-segment-slope}, the corresponding multiset for
$\NP(f_A)$ is $\{\eta_1,\ldots,\eta_n\}$.  Equality of the two polygons therefore
implies (i).  Conversely, (i) gives the same ordered slopes with the same
horizontal lengths for both polygons.  Both polynomials are monic, so both
polygons have the right endpoint $(n,0)$.  Their slopes, lengths, and this common
endpoint determine the same polygonal line.
\end{proof}
The following theorem characterizes spectral equality through five equivalent
conditions.
\begin{theorem}
\label{thm:spectral-equality-characterization}
Under the standing hypotheses, the following statements are equivalent:
\begin{enumerate}[label=\textup{(\roman*)}]
 \item the eigenvalues of $\mathbf A$ can be ordered so that
 \[
  \val(\rho_i)=\eta_i,
  \qquad i=1,\ldots,n;
 \]
 \item $\NP(p_{\mathbf A})=\NP(f_A)$;
 \item for every $k\in\mathcal B(A)$,
 \[
  \val(\boldsymbol\alpha_k)=a_k;
 \]
 \item for every $k\in\mathcal B(A)$, there exists
 $I\in\mathcal I_k(A)$ such that
 \[
  \val\!\left(\det\mathbf A[I,I]\right)
  =\per(A[I,I])=a_k;
 \]
 \item for every $k\in\mathcal B(A)$, there exists
 $I\in\mathcal I_k(A)$ such that
 \[
  \Delta_{A[I,I]}\!\left(\mathbf A[I,I]\right)\neq0.
 \]
\end{enumerate}
\end{theorem}

\begin{proof}
Proposition~\ref{prop:newton-spectral-equality} gives
(i)$\Longleftrightarrow$(ii), and
Proposition~\ref{prop:newton-polygon-equality} gives
(ii)$\Longleftrightarrow$(iii).  For every $k\in\mathcal B(A)$,
Proposition~\ref{prop:coefficient-realization} gives
(iii)$\Longleftrightarrow$(iv)$\Longleftrightarrow$(v).
\end{proof}

\begin{corollary}
\label{cor:distinct-diagonal}
If
\[
 \eta_1>\eta_2>\cdots>\eta_n,
\]
then $\mathcal B(A)=\{1,\ldots,n\}$.  Therefore, spectral equality holds if
and only if, for every $k=1,\ldots,n$, there exists
$I\in\mathcal I_k(A)$ such that
\[
 \Delta_{A[I,I]}\!\left(\mathbf A[I,I]\right)\neq0.
\]
\end{corollary}

\begin{corollary}
\label{cor:sign-nonsingular-spectral-equality}
Suppose that, for every $k\in\mathcal B(A)$, there exists
$I_k\in\mathcal I_k(A)$ such that $A[I_k,I_k]$ is tropically
sign-nonsingular.  Then the eigenvalues of $\mathbf A$ can be ordered so that
\[
 \val(\rho_i)=\eta_i,
 \qquad i=1,\ldots,n.
\]
\end{corollary}

\begin{proof}
Proposition~\ref{prop:det-per} and
Lemma~\ref{lem:dominant-cancellation} give
\[
 \Delta_{A[I_k,I_k]}\!\left(\mathbf A[I_k,I_k]\right)\neq0
 \qquad(k\in\mathcal B(A)).
\]
Condition (v) of Theorem~\ref{thm:spectral-equality-characterization} applies.
\end{proof}

The last corollary strengthens the earlier sufficient criterion in two ways.
Only the breakpoint indices $k\in\mathcal B(A)$ need to be tested, and at each
such index one may use any member of $\mathcal I_k(A)$, rather than a fixed
principal submatrix selected in advance.
\begin{remark}
The general Newton polygon principle already implies that equality of the
classical and tropical coefficient valuations yields equality of the root
valuations.  Theorem~\ref{thm:spectral-equality-characterization} identifies
the obstruction for totally nonnegative lifts.  It is dominant cancellation
in the principal minors that give the essential tropical coefficients.  The
theorem also shows that, when diagonal values are repeated, we only need to
check the indices in $\mathcal B(A)$.
Corollary~\ref{cor:sign-nonsingular-spectral-equality} gives a simpler tropical
sufficient condition.
\end{remark}
It is useful to compare this characterization with the earlier results.
The generic theorem of Akian--Bapat--Gaubert
\cite[Theorem~7.4]{AkianBapatGaubert} concerns an unrestricted parameter
space and asserts equality outside a proper algebraic exceptional set.
The result of Gaubert--Niv
\cite[Proposition~6.5]{GaubertNiv} gives equality under tropical and
classical total positivity. Theorem~
\ref{thm:spectral-equality-characterization} instead gives an exact
pointwise criterion inside the space of totally nonnegative lifts.
In particular, it applies to lifts that are not totally positive and
also detects the totally nonnegative lifts for which spectral equality
fails.
\subsection{Exceptional loci in the totally nonnegative lift space}

We now study the admissible leading-coefficient space. Unlike the
unrestricted parameter space, this space may be reducible. Some of its
irreducible components may lie entirely in a dominant-cancellation set.
Therefore, generic spectral equality must be studied separately on each
irreducible component.

Let
\[
 E(A)=\{(i,j):a_{ij}\neq-\infty\},\qquad N=|E(A)|.
\]
For a finite entry of $A$, write
\[
 \mathbf a_{ij}=c_{ij}t^{a_{ij}}+\text{terms of smaller valuation},
 \qquad c_{ij}>0,
\]
and define the admissible leading-coefficient space by
\[
 \mathcal C_{\TN}(A)
 =\left\{c\in\mathbb R_{>0}^{E(A)}:
 \begin{array}{l}
 \text{there exists }\mathbf A\in\TN(\K)^{n\times n}\text{ such that}\\
 \val(\mathbf A)=A\text{ and }
 \operatorname{lc}(\mathbf a_{ij})=c_{ij}
 \text{ for every }(i,j)\in E(A)
 \end{array}\right\}.
\]
Thus the lower-order terms may vary, while the support,
valuations, and leading coefficients are prescribed.  In the formulas below
we put $c_{ij}=0$ when $(i,j)\notin E(A)$.  All Zariski closures,
irreducible components, and vanishing ideals in this subsection are taken
over $\mathbb R$.  For
$I\in\mathcal I_k(A)$, set
\[
 \Delta_I(c)
 =\sum_{\sigma\in\operatorname{Opt}(A[I,I])}
   \operatorname{sgn}(\sigma)\prod_{i\in I}c_{i,\sigma(i)},
\]
where $\operatorname{Opt}(A[I,I])$ is the set of maximizing permutations of
$A[I,I]$, and define
\begin{equation}\label{eq:Dk-definition}
 D_k(c)=\sum_{I\in\mathcal I_k(A)}\Delta_I(c).
\end{equation}
The following theorem describes the exceptional locus and gives the
componentwise dichotomy for spectral equality.
\begin{theorem}
\label{thm:generic-spectral-equality}
Under the standing hypotheses, the following statements hold.
\begin{enumerate}[label=\textup{(\roman*)}]
 \item A totally nonnegative lift with leading-coefficient vector $c$ has
 spectral equality if and only if
 \[
  D_k(c)\neq0\qquad\text{for every }k\in\mathcal B(A).
 \]
 \item The exceptional set on which spectral equality fails is
 \begin{equation}\label{eq:exceptional-set}
  \mathcal E(A)=\mathcal C_{\TN}(A)\cap
  \bigcup_{k\in\mathcal B(A)}V(D_k).
 \end{equation}
 Thus, relative to the admissible parameter space, failure is described by
 finitely many polynomial equations in the leading coefficients.
 \item Put $D=\prod_{k\in\mathcal B(A)}D_k$, and let $Z$ be an irreducible
 component of the Zariski closure of $\mathcal C_{\TN}(A)$.  Exactly one of
 the following alternatives occurs:
 \begin{enumerate}[label=\textup{(\alph*)}]
  \item $Z$ is spectrally obstructed: for some $k\in\mathcal B(A)$,
  \[
   D_k|_Z\equiv0.
  \]
  Every admissible lift with leading-coefficient vector in $Z$ then fails to
  have spectral equality.
  \item $Z$ is spectrally nonobstructed: $D|_Z\not\equiv0$.  Spectral equality
  then holds at every admissible point of the Zariski-open dense subset
  $Z\setminus V(D)$; equivalently, it holds on
  \[
   \mathcal C_{\TN}(A)\cap\bigl(Z\setminus V(D)\bigr).
  \]
  No density assertion for this intersection inside the admissible part of
  $Z$ is needed.
 \end{enumerate}
 Also, if $I(Z)$ is the prime ideal of $Z$, then $Z$ is spectrally
 obstructed if and only if $D_k\in I(Z)$ for at least one
 $k\in\mathcal B(A)$.
\end{enumerate}
\end{theorem}

\begin{proof}
For $I\in\mathcal I_k(A)$, the terms of valuation $a_k$ in
$\det\mathbf A[I,I]$ have total coefficient $\Delta_I(c)$.  Since all
principal minors of $\mathbf A$ are nonnegative, each $\Delta_I(c)$ is
nonnegative, and principal minors belonging to distinct index sets cannot
cancel one another.  Therefore, the coefficient of $t^{a_k}$ in
$\boldsymbol\alpha_k$ is $D_k(c)$, and
\[
 \val(\boldsymbol\alpha_k)=a_k
 \quad\Longleftrightarrow\quad D_k(c)\neq0.
\]
Part~(i) follows from
Theorem~\ref{thm:spectral-equality-characterization}, and part~(ii) follows
directly.

For part~(iii), irreducibility of the real algebraic set $Z$ implies that its
vanishing ideal $I(Z)$ is prime.  Indeed, if $fg$ vanishes on $Z$, then
$Z\subseteq V(f)\cup V(g)$, so irreducibility gives
$Z\subseteq V(f)$ or $Z\subseteq V(g)$.  If $D|_Z\equiv0$, then
\[
 \prod_{k\in\mathcal B(A)}D_k\in I(Z),
\]
so primality gives $D_k\in I(Z)$ for at least one breakpoint $k$.  Part~(i)
then shows that every admissible point of $Z$ is spectrally exceptional.
Conversely, $D_k\in I(Z)$ clearly implies $D|_Z\equiv0$.  If
$D|_Z\not\equiv0$, then $V(D)\cap Z$ is a proper Zariski-closed subset of
$Z$.  Thus $Z\setminus V(D)$ is Zariski open and dense in $Z$, and part~(i)
identifies all its admissible points with lifts having spectral equality.
\end{proof}

The next corollary turns the obstruction criterion into a Gr\"obner-basis
membership test.
\begin{corollary}
\label{cor:ideal-obstruction-test}
Suppose that the prime ideal $I(Z)$ of the irreducible component $Z$ is known
by generators, and let $G_Z$ be a Gr\"obner basis of $I(Z)$ with respect to a
fixed monomial order.  Then $Z$ is
spectrally obstructed if and only if, for some $k\in\mathcal B(A)$, the
remainder of $D_k$ upon division by $G_Z$ is zero.
\end{corollary}

\begin{proof}
The remainder is zero exactly when $D_k\in I(Z)$.  Apply
Theorem~\ref{thm:generic-spectral-equality}(iii).
\end{proof}

\begin{remark}
Corollary~\ref{cor:ideal-obstruction-test} is an effective ideal-membership
test once $I(Z)$ has been computed.  It is not an algorithm for finding the
irreducible components of the Zariski closure of $\mathcal C_{\TN}(A)$; that
elimination problem is separate and may be the computationally difficult
part.
\end{remark}

\begin{remark}
Let $K_0$ be a compact subset of
\[
 \mathcal C_{\TN}(A)\setminus
 \bigcup_{k\in\mathcal B(A)}V(D_k).
\]
There exists $\delta>0$ such that
\[
 D_k(c)\geq\delta
 \qquad(c\in K_0,\ k\in\mathcal B(A)).
\]
Therefore, spectral equality remains valid under all sufficiently small
perturbations of the leading coefficients that remain in the admissible lift
space and sufficiently close to $K_0$.  Indeed, every $D_k$ is continuous and
strictly positive on $K_0$: nonvanishing and
total nonnegativity together give positivity.  Compactness and the finiteness
of $\mathcal B(A)$ yield a common positive minimum $\delta$.  Continuity then
gives the perturbation statement.
\end{remark}
The following family contains both ordinary lifts with spectral equality and
exceptional totally nonnegative lifts without it.
\begin{example}
\label{ex:parametric-exceptional-locus}
Consider
\[
 A=\begin{pmatrix}
 10&9&4\\
 9&8&4\\
 4&4&2
 \end{pmatrix}.
\]
The matrix \(A\) belongs to 
$(\TNT)^{3\times3}\setminus(\TPT)^{3\times3}$ and, since
$10>8>2$, we have $\mathcal B(A)=\{1,2,3\}$.

For $r,s>0$, consider
\[
 \mathbf A(r,s)=\begin{pmatrix}
 t^{10}&s t^9&t^4\\
 t^9&r t^8&t^4\\
 t^4&t^4&t^2
 \end{pmatrix}.
\]
In this two-parameter slice, the dominant coefficient polynomials are
\[
 D_1(r,s)=1,\qquad D_2(r,s)=r-s,\qquad D_3(r,s)=r-s.
\]
Thus the exceptional algebraic locus is the line $r=s$.  If $r>s>0$, the
only tied $2\times2$ minor is
\[
 \det\begin{pmatrix}t^{10}&s t^9\\t^9&r t^8\end{pmatrix}
 =(r-s)t^{18}>0,
\]
and every other $2\times2$ minor has a unique positive dominant term.
Also,
\[
 \det\mathbf A(r,s)
 =(r-s)t^{20}-t^{18}+(s+1)t^{17}-r t^{16}>0.
\]
Thus $\mathbf A(r,s)$ is totally positive and has spectral equality whenever
$r>s>0$.

The exceptional line also contains leading-coefficient vectors of totally
nonnegative lifts.  For $0<\varepsilon<2$, set
\[
 \mathbf A_\varepsilon=\begin{pmatrix}
 t^{10}&t^9&t^4\\
 t^9&t^8+t^{8-\varepsilon}&t^4\\
 t^4&t^4&t^2
 \end{pmatrix}.
\]
Its upper-left principal minor is
\[
 t^{10}(t^8+t^{8-\varepsilon})-t^{18}
 =t^{18-\varepsilon}>0.
\]
All other $2\times2$ minors have a unique positive dominant term, and
\[
 \det\mathbf A_\varepsilon
 =t^{20-\varepsilon}-t^{18}+2t^{17}-t^{16}-t^{16-\varepsilon}>0.
\]
Therefore $\mathbf A_\varepsilon\in\TN(\K)^{3\times3}$.  Nevertheless,
\[
 \val\!\left(\det\mathbf A_\varepsilon[\{1,2\},\{1,2\}]\right)
 =18-\varepsilon<18
\]
and
\[
 \val(\det\mathbf A_\varepsilon)=20-\varepsilon<20.
\]
Thus spectral equality fails for this exceptional lift.  The example shows
that the cancellation set can meet the totally nonnegative class.  It also
shows why Theorem~\ref{thm:generic-spectral-equality} must be stated separately
on each irreducible component.
\end{example}
\subsection{A complete characterization in dimension three}

Assume now that $n=3$, and let
\[
 d_1\geq d_2\geq d_3
\]
be the decreasing rearrangement of $a_{11},a_{22},a_{33}$.  For
$1\leq i<j\leq3$, put
\begin{equation}\label{eq:Delta-ij}
 \Delta_{ij}
 =c_{ii}c_{jj}
 -\mathbf1_{\{a_{ii}+a_{jj}=a_{ij}+a_{ji}\}}c_{ij}c_{ji},
\end{equation}
and define
\begin{equation}\label{eq:D2-three}
 \mathcal P_2(A)=
 \bigl\{\{i,j\}:a_{ii}+a_{jj}=d_1+d_2\bigr\},
 \qquad
 D_2(c)=\sum_{\{i,j\}\in\mathcal P_2(A)}\Delta_{ij}.
\end{equation}
Finally, for $w_\sigma(A)=\sum_i a_{i,\sigma(i)}$, set
\begin{equation}\label{eq:Delta-123}
 \Delta_{123}
 =\sum_{\substack{\sigma\in S_3\\
 w_\sigma(A)=d_1+d_2+d_3}}
 \operatorname{sgn}(\sigma)\prod_{i=1}^3c_{i,\sigma(i)}.
\end{equation}

For reference, all principal minors that can contribute in dimension three
and their dominant data are displayed below.  A row of order $k$ contributes
to the $k$th characteristic coefficient at the tropical leading weight only
when the middle entry equals $d_1+\cdots+d_k$.
\begin{equation}\label{eq:three-principal-minor-table}
\begin{array}{c|c|c}
 I&\operatorname{per}(A[I,I])&
 \text{coefficient at the dominant weight}\\
\hline
\{1\}&a_{11}&c_{11}\\
\{2\}&a_{22}&c_{22}\\
\{3\}&a_{33}&c_{33}\\
\{1,2\}&a_{11}+a_{22}&\Delta_{12}\\
\{1,3\}&a_{11}+a_{33}&\Delta_{13}\\
\{2,3\}&a_{22}+a_{33}&\Delta_{23}\\
\{1,2,3\}&d_1+d_2+d_3&\Delta_{123}
\end{array}
\end{equation}
The middle column uses the Monge property: the identity permutation maximizes
every principal submatrix.  In particular, the table retains all three
principal minors of order two rather than selecting only one maximizer.

The following corollary gives a complete spectral-equality criterion in
dimension three.
\begin{corollary}
\label{cor:complete-three-by-three}
Let $\mathbf A\in\TN(\K)^{3\times3}$, let $A=\val(\mathbf A)$, and suppose
that the diagonal entries of $A$ are finite.  Then the eigenvalues of
$\mathbf A$ can be ordered so that
\[
 \val(\rho_i)=d_i,\qquad i=1,2,3,
\]
if and only if
\begin{equation}\label{eq:three-det-condition}
 \Delta_{123}>0
\end{equation}
and, whenever $d_2>d_3$,
\begin{equation}\label{eq:three-second-condition}
 D_2(c)>0.
\end{equation}
In other words, if $d_2>d_3$, both
\eqref{eq:three-det-condition} and \eqref{eq:three-second-condition} are
required; if $d_2=d_3$, condition \eqref{eq:three-det-condition} alone is
necessary and sufficient.
\end{corollary}

\begin{proof}
The breakpoint set is
\[
 \mathcal B(A)=
 \{1:d_1>d_2\}\cup\{2:d_2>d_3\}\cup\{3\}.
\]
The first classical coefficient is the trace of $\mathbf A$.  Its summands are
nonnegative, so
\[
 \val(\boldsymbol\alpha_1)=d_1;
\]
hence the breakpoint $k=1$, when present, creates no additional condition.

For $k=2$, the maximizing principal index sets are exactly the members of
$\mathcal P_2(A)$.  The Monge inequalities give
$a_{ii}+a_{jj}\geq a_{ij}+a_{ji}$, and the dominant coefficient of the
corresponding principal minor is \eqref{eq:Delta-ij}.  Total nonnegativity
gives $\Delta_{ij}\geq0$, so
\[
 \val(\boldsymbol\alpha_2)=d_1+d_2
 \quad\text{if and only of}\quad D_2(c)>0.
\]
This condition is spectrally essential exactly when $d_2>d_3$.

For $k=3$, the only principal index set is $\{1,2,3\}$.  Since the identity
permutation maximizes a Monge matrix, its tropical permanent is
$d_1+d_2+d_3$, and its dominant determinant coefficient is
$\Delta_{123}$.  Therefore
\[
 \val(\det\mathbf A)=d_1+d_2+d_3
 \quad\text{if and only of}\quad\Delta_{123}\neq0.
\]
Total nonnegativity gives $\Delta_{123}\geq0$, so nonvanishing is equivalent
to strict positivity.  The result follows from
Theorem~\ref{thm:spectral-equality-characterization}.
\end{proof}

The criterion takes the following explicit form for the four possible patterns
of the diagonal entries.
\begin{corollary}
\label{cor:three-by-three-diagonal-patterns}
Under the hypotheses of Corollary~\ref{cor:complete-three-by-three}, the four
possible diagonal patterns give
\[
\begin{array}{c|c|c}
\text{diagonal pattern}&\mathcal B(A)&\text{spectral equality conditions}\\
\hline
d_1>d_2>d_3&\{1,2,3\}&D_2(c)>0,\ \Delta_{123}>0\\
d_1=d_2>d_3&\{2,3\}&D_2(c)>0,\ \Delta_{123}>0\\
d_1>d_2=d_3&\{1,3\}&\Delta_{123}>0\\
d_1=d_2=d_3&\{3\}&\Delta_{123}>0.
\end{array}
\]
In particular, when $d_2=d_3$, dominant cancellation in the second
coefficient does not affect the spectrum because the corresponding tropical
coefficient point is not a strict Newton-polygon vertex.
\end{corollary}

\section{Concluding remarks}

Tropical total nonnegativity, its Monge characterization, and non-Archimedean
valuation give a simple description of the algebraic tropical spectrum.  The
weak majorization inequality for eigenvalue valuations also has a direct
multiplicative form as weak log-majorization.

Theorem~\ref{thm:spectral-equality-characterization} gives the exact condition
for equality beyond total positivity.  At every essential breakpoint, at
least one maximizing principal minor must have a nonzero dominant
coefficient.  Thus the main issue is dominant cancellation in the relevant
principal minors, not sign-nonsingularity alone.

Corollary~\ref{cor:sign-nonsingular-spectral-equality} gives a simple
sufficient condition that can be checked directly.
Example~\ref{ex:parametric-exceptional-locus} shows that this condition is not
necessary.  Theorem~\ref{thm:generic-spectral-equality} describes the failure
set by polynomial equations on each irreducible component of the admissible
parameter space.  Corollary~\ref{cor:complete-three-by-three} gives an explicit
test in dimension three, and Example~\ref{ex:parametric-exceptional-locus}
shows both possible cases in one family.

Our results suggest two directions for further study.  First, it would
be useful to describe, directly in terms of the tropical matrix \(A\),
the irreducible components of the Zariski closure of
\(\mathcal C_{\TN}(A)\).  The ideal-membership test given above assumes
that the ideals of these components are already known and does not
provide a method for computing them.  Second, one may seek quantitative
bounds for the changes in the eigenvalue valuations caused by dominant
cancellation, for example in terms of the coefficient drops
\[
 a_k-\val(\boldsymbol\alpha_k)
\]
or the distance of the leading-coefficient vector from the exceptional
locus.  We do not pursue these questions here.

\section*{Acknowledgment}
The authors acknowledge the support of the Department of Mathematical Sciences
at Isfahan University of Technology, Iran.


\begin{thebibliography}{99}

\bibitem{AkianBapatGaubert}
M.~Akian, R.~B. Bapat, and S.~Gaubert,
\emph{Non-Archimedean valuations of eigenvalues of matrix polynomials},
Linear Algebra Appl. \textbf{498} (2016), 592--627.

\bibitem{AkianGaubertSharify}
M.~Akian, S.~Gaubert, and M.~Sharify,
\emph{Log-majorization of the moduli of the eigenvalues of a matrix polynomial
by tropical roots},
Linear Algebra Appl. \textbf{528} (2017), 394--435.

\bibitem{Butkovic}
P.~Butkovi\v{c},
\emph{Max-linear Systems: Theory and Algorithms},
Springer Monographs in Mathematics, Springer, London, 2010.

\bibitem{FallatJohnson}
S.~M. Fallat and C.~R. Johnson,
\emph{Totally Nonnegative Matrices},
Princeton University Press, Princeton, NJ, 2011.

\bibitem{GaubertNiv}
S.~Gaubert and A.~Niv,
\emph{Tropical totally positive matrices},
J. Algebra \textbf{515} (2018), 511--544.

\bibitem{Marker}
D.~Marker,
\emph{Model Theory: An Introduction},
Graduate Texts in Mathematics, vol.~217, Springer-Verlag, New York, 2002.

\bibitem{Markwig}
T.~Markwig,
\emph{A field of generalised Puiseux series for tropical geometry},
Rend. Semin. Mat. Univ. Politec. Torino \textbf{68} (2010), no.~1, 79--92.

\bibitem{MarshallOlkinArnold}
A.~W. Marshall, I.~Olkin, and B.~C. Arnold,
\emph{Inequalities: Theory of Majorization and Its Applications},
2nd ed., Springer Series in Statistics, Springer, New York, 2011.

\bibitem{ManjeganiPointSpectrum}
S.~M. Manjegani, A.~Peperko, and H.~Shokooh Saljooghi,
\emph{Asymptotic formulae and inequalities for point spectrum in max algebra},
Linear Algebra Appl. \textbf{634} (2022), 112--136.

\bibitem{ManjeganiParsa}
S.~M. Manjegani and T.~Parsa,
\emph{Stochastic matrices and majorization in max algebra},
Linear Algebra Appl. \textbf{736} (2026), 225--244.

\bibitem{peper}
A.~Rosenmann, F.~Lehner, and A.~Peperko,
\emph{Polynomial convolutions in max-plus algebra},
Linear Algebra Appl. \textbf{578} (2019), 370--401.
\bibitem{Ag}
A.~Agudelo and O.~Lorscheid,
\emph{Factorizations of tropical and sign polynomials},
Indag. Math. (N.S.) \textbf{32} (2021), no.~4, 797--812.

\end{thebibliography}
\end{document}